\documentclass[12pt,reqno]{amsart}

\usepackage[utf8]{inputenc}
\usepackage[english]{babel}
\usepackage[T1]{fontenc}
\usepackage{graphicx}
\usepackage{amssymb,amstext,amsthm}
\usepackage{enumerate}
\usepackage{geometry}
\usepackage{amsmath,amscd}
\usepackage{hyperref}
\usepackage[all]{xy}
\usepackage[normalem]{ulem}
\usepackage{calc,color}
\usepackage{epsfig,graphics}
\usepackage{pst-grad}
\usepackage{pst-plot}
\usepackage{footnote}
\usepackage[pagewise]{lineno}

\newcommand{\R}{\ensuremath{\mathbb{R}}}
\newcommand{\Q}{\ensuremath{\mathbb{Q}}}

\newcommand{\LL}{\ensuremath{\mathcal{L}}}
\newcommand{\delim}[3]{\left#1 #3 \right#2}

\newcommand{\pin}[1]{\delim{\langle}{\rangle}{#1}}

\def \ds {\displaystyle}
\def \tn {\textnormal}

\newtheorem{theorem}{Theorem}[section]
\newtheorem{lemma}[theorem]{Lemma}

\newtheorem{remark}[theorem]{Remark}
\numberwithin{equation}{section}

\title[Large Diffusivity]{Large Diffusivity and Rate of Convergence of Attractors in Parabolic Systems}

\author[L. Pires]{Leonardo Pires}
\address[L. Pires]{Departamento  de Matem\'atica e Estat\'istica,
Universidade Estadual de Ponta Grossa, 84030-900 Ponta Grossa PR, Brazil}
\email{lpires@uepg.br}
\begin{document}

\begin{abstract}
In this paper we are concerned with rate of convergence of parabolic systems with large diffusion. We will exhibit the exact moment that spatial homogenization occurs and estimate the continuity of attractors by a  rate of convergence. We will show an example where our estimate is optimal.
\end{abstract}

\maketitle
\thispagestyle{empty}

\section{Introduction}
Many reaction diffusion equations originated by models of heat diffusion does not display the creation of stable patterns, that is, stable solutions which are not spatially dependent in an adequate limit process involving parameters. Diffusive processes where solutions have this spatial homogenization  have been studied in the works \cite{Conley1978} and \cite{Hale1986}.

Here we consider a system of parabolic equation with large diffusion in all domain which the limiting problem is an ordinary differential equation in $\R^n$. More precisely we will impose conditions in the limiting ODE system in order to ensure that the PDE has a global attractor converging to the limiting atractor with a precise rate of convergence. Hence in this paper we generalize some results obtained in the above works presenting a estimate of how fast can be the spatial homogeneity and exhibiting  the exact moment that this phenomenon occurs.

To state our results let $\Omega$ be a bounded open set in $\R^N$, $N\leq 3$, with boundary $\Gamma=\partial \Omega$ smooth and consider the system of reaction-diffusion equations of the form
\begin{equation}\label{reaction_diffusion_equation}
\begin{cases}
u^\varepsilon_t-E\Delta u^\varepsilon+u^\varepsilon=F(u^\varepsilon),&\,\,t>0,\,\,\,x\in\Omega,\\ \frac{\partial u^\varepsilon}{\partial\vec{n}}=0, &\,\,t>0,\,\,\,x\in\Gamma,
\end{cases}
\end{equation} 
where $u^\varepsilon=(u^\varepsilon_1,...,u^\varepsilon_n)\in\R^n$, $E=\tn{diag}(\varepsilon_1,...,\varepsilon_n)$, with $\varepsilon_i\geq m_0>0$, $i=1,...,n$,  $\vec{n}$ is the outward normal vector to the $\Gamma$ and $\frac{\partial u^\varepsilon}{\partial\vec{n}}=(\pin{\nabla u^\varepsilon_1,\vec{n}},\dots,\pin{\nabla u^\varepsilon_n,\vec{n}})$. We assume that the nonlinearity $F:\R^n\to \R^n$ is bounded continuously differentiable and satisfies other hypotheses stated later.  

We will see when $d_\varepsilon:=\ds\min_{i=1,...,n}{\varepsilon_i}\to \infty$  the solutions of \eqref{reaction_diffusion_equation}  converge to a solution of the following ordinary differential equation
\begin{equation}\label{limiting_ODE_intro}
\dot{u}^\infty(t)+u^\infty(t)=F(u^\infty(t)),
\end{equation}   
where $u^\infty(t)\in \mathbb{R}^n$ and by simplicity we have assumed $|\Omega|=1$.

Under standard conditions the equations \eqref{reaction_diffusion_equation} and \eqref{limiting_ODE_intro} are globally well posed in a Hilbert space $X_\varepsilon^\frac{1}{2}$ and $\mathbb{R}^n$ respectively. Moreover the nonlinear semigroup generated by its solutions have a global attractor $\mathcal{A}_\varepsilon\subset X_\varepsilon^\frac{1}{2}$ and $\mathcal{A}_\infty\subset\mathbb{R}^n$. If we consider $\mathbb{R}^n$ embedding in $X_\varepsilon^\frac{1}{2}$ as the constant functions, the main result of this paper states
\begin{equation}\label{Hausdorff_A}
\tn{d}_H(\mathcal{A}_\varepsilon,\mathcal{A}_\infty)\leq  \frac{C}{\sqrt{d_\varepsilon}},
\end{equation}
for $d_\varepsilon\ $ in a appropriate bounded interval, $\tn{d}_H$ denotes the Hausdorff distance between sets in $X_\varepsilon^\frac{1}{2}$ and $C$ denotes a constant independent of $d_\varepsilon$. It was showed in \cite{Hale1986}, for $d_\varepsilon$ sufficienty large we have $\mathcal{A}_\varepsilon=\mathcal{A}_\infty$. We will calculate the upper limit value $\mu$ when this fact begins to occur. Therefore \eqref{Hausdorff_A} ensure the continuity of the family $\{\mathcal{A}_\varepsilon\}_\varepsilon$ as $\varepsilon\to \mu^-$. Basically the $\omega-$limit set set of every solutions of \eqref{reaction_diffusion_equation} lie in a bounded set of $\mathbb{R}^n$ and it must be a union of invariant sets of \eqref{limiting_ODE_intro} which belongs to this bounded set. But such invariant sets must belong to $\mathcal{A}_\infty$ and it is clear that for $d_\varepsilon$ sufficiently large this bounded set will be $\mathcal{A}_\varepsilon$.

But we go further, we will show the existence of an invariant manifold $\mathcal{M}_\varepsilon$ for \eqref{reaction_diffusion_equation} containing the global attractor $\mathcal{A}_\varepsilon$. Notice that trivially $\mathbb{R}^n$ is a invariant manifold for \eqref{limiting_ODE_intro} containing $\mathcal{A}_\infty$. We will show that, in some sense, $\mathcal{M}_\varepsilon$ approaches $\mathbb{R}^n$ with the same rate $1/\sqrt{d_\varepsilon}$.

The values of $\lambda$ such that 
\begin{equation}
\begin{cases}
E\Delta u^\varepsilon+u^\varepsilon=\lambda u^\varepsilon,&x\in\Omega,\\ \frac{\partial u^\varepsilon}{\partial\vec{n}}=0, &x\in\Gamma,
\end{cases}
\end{equation}
has a non zero solutions are called eigenvalues. We will see that they are real numbers and can be ordered  in the following way $\{1<\lambda_2^\varepsilon<\lambda_3^\varepsilon,\dots\}$. Moreover  we have  $\lambda_j^\varepsilon\to \infty$ as $d_\varepsilon\to \infty$ for $j\geq 2$. Thus we can consider the spectral projection $Q_\varepsilon$ whose the image can be identified with $\mathbb{R}^n$. We will prove that 
$$
\|Q_\varepsilon-P\|_{\LL(L^2(\Omega,\mathbb{R}^n),X_\varepsilon^\frac{1}{2})}\leq \frac{C}{\sqrt{d_\varepsilon}},
$$
where $C$ is a constant independent of $d_\varepsilon$ and $P$ denotes the average projection on $\Omega$.

An interesting question arises when we ask if the exponent $-1/2$ in the above estimate is optimal. We will exhibit an example where the convergence of resolvent operators in exact $Cd_\varepsilon^{-\frac{1}{2}}$. It is well known and we will see in this work that this convergence of resolvent operator will imply the convergence of global attractors and invariant manifolds.

This paper is divided as follows: in Section 2 we present the functional phase space to deal with \eqref{reaction_diffusion_equation} and \eqref{limiting_ODE_intro} and we state conditions to ensure the existence of global attractors. In Section 3 we make precise in what sense the spatial homogenization occurs. In Section 4 we deal with the spectral convergence and we obtain the rate of convergence for the resolvent operators. In Section 4 we prove the main result of this work concerning rate of convergence of attractors.

\section{Functional Setting}
The phase space to deal with system of reaction diffusion equation as \eqref{reaction_diffusion_equation} is generally the Sobolev space $H^1(\Omega,\R^n)$, but since we have the diffusion coefficient as the parameter $\varepsilon_i$ is natural to consider a metric with some weight and use the fractional power spaces associated with sectorial operators. They play a key role in the theory of the existence of solutions to nonlinear partial differential equations of the parabolic type and in the analysis of asymptotic behavior of its solutions. 

Consider the operator $A_\varepsilon=\tn{diag}(A_1,...,A_n)$, where $A_i:D(A_i)\subset L^2(\Omega)\to L^2(\Omega)$, $i=1,...,n$, is given by 
\begin{equation}\label{operator_reaction_diffusion}
\begin{cases}
D(A_i)=\{\varphi \in H^1(\Omega):\frac{\partial \varphi}{\partial\vec{n}}=0,\tn{ in }\Gamma\},\\A_i\varphi=-\varepsilon_i\Delta \varphi+\varphi.
\end{cases}
\end{equation}

Let $A_i^\alpha$ be the fractional power of operator $A_i$ and denote $X_i^\alpha$  its fractional power space  endowed with the graph norm. If $N\leq 3$ and  $\frac{3}{4}<\alpha<1$, according to \cite{Henry1980}, we have $X^\alpha \hookrightarrow H^1(\Omega,\mathbb{R}^n)\cap L^\infty(\Omega,\mathbb{R}^n)$ with continuous inclusion and
\begin{equation}\label{domain_reaction_diffusion_times}
D(A_\varepsilon^\frac{1}{2})= X_1^\frac{1}{2} \times\dots\times X_n^\frac{1}{2}=H^1(\Omega,\R^n).
\end{equation}
Thus we take as phase space for \eqref{reaction_diffusion_equation} the space $X_\varepsilon^\frac{1}{2}=H^1(\Omega,\R^n)$ with the inner product given by
\begin{equation}\label{pin_reaction_diffusion}
\pin{\varphi,\psi}_{X_\varepsilon^\frac{1}{2}}=\int_\Omega E\nabla\varphi\nabla\psi\,dx+\int_\Omega \varphi\psi\,ds,\quad \varphi,\psi\in X_\varepsilon^\frac{1}{2},
\end{equation}
and we rewrite \eqref{reaction_diffusion_equation} in the abstract  form
\begin{equation}\label{coupled_system_reaction_diffusion}
\begin{cases}
u^\varepsilon_t+A_\varepsilon u^\varepsilon=f(u^\varepsilon),\,\,t>0,\\ u^\varepsilon(0)=u^\varepsilon_0\in X_\varepsilon^\frac{1}{2},
\end{cases}
\end{equation}
where $f:X_\varepsilon^\frac{1}{2}\to L^2(\Omega,\R^n)$ is given by $f(u)(x)=F(u(x))$, for $u\in X_\varepsilon^\frac{1}{2}$.

To obtain the well posedness  and  existence of the global attractor for  equation \eqref{coupled_system_reaction_diffusion}, we need to impose some growth and dissipativeness conditions, these conditions are statement in \cite{J.M.Arrieta1999,J.M.Arrieta2000} and \cite{A.N.Carvalho2010}. 
\begin{itemize}
\item[(i)] \textbf{Growth condition}. If $n=2$, for all $\eta>0$, there is a constant $C_\eta>0$ such that
$$
 |f(u)-f(v)|\leq C_\eta (e^{\eta|u|^2}+e^{\eta|v|^2})|u-v|,\quad \forall\,u,v\in \R,
$$     
and if $n\geq 3$, there is a constant $\tilde{C}>0$ such that
$$
|f(u)-f(v)|\leq \tilde{C}|u-v| (|u|^{\frac{4}{n-2}}+|v|^{\frac{4}{n-2}}+1 ),\quad \forall\,u,v\in \R.
$$     
\item[(ii)]\textbf{Dissipativeness condition} 
$$
\limsup_{|u|\to\infty} \dfrac{f(u)}{u} <0.
$$
\end{itemize}

The theory of well-posedness of abstract parabolic problems that enable us to study \eqref{coupled_system_reaction_diffusion} is developed in \cite{A.N.Carvalho2010}. Results in local well-posedness in the energy space $X_\varepsilon^\frac{1}{2}$ are obtained due the fact that $A_\varepsilon$ generates a strong continuous semigroup and in addition $f$ is continuously differentiable satisfying the above growth condition (i). To show that all solutions of \eqref{coupled_system_reaction_diffusion} are globally defined, we need to impose the above  dissipativeness condition (ii). Thus, for each initial date $u_0^\varepsilon$ in $X_\varepsilon^\frac{1}{2}$, the equation \eqref{coupled_system_reaction_diffusion} has a global solution through $u_0^\varepsilon$. This solution is continuously differentiable with respect to the initial data and it is a classical solution for $t>0$ satisfying the variation of constants formula. Moreover \eqref{coupled_system_reaction_diffusion} has a global attractor uniformly bounded in $\varepsilon_i$, $i=1,...,n$.

Thus we assume the existence of the solutions $u^\varepsilon(t,u^\varepsilon_0)$ of \eqref{coupled_system_reaction_diffusion} through $u^\varepsilon_0\in X_\varepsilon^\frac{1}{2}$ for positive time and the nonlinear semigroup defined by $T_\varepsilon(t)u_0^\varepsilon=u^\varepsilon(t,u^\varepsilon_0)$ satisfies the variation of constants formula
\begin{equation}\label{vcf1_rd}
T_\varepsilon(t)u_0^\varepsilon=e^{-A_\varepsilon t}u_0^\varepsilon+\int_0^te^{-A_\varepsilon (t-s)}f(T_\varepsilon(s)u_0^\varepsilon)\,ds,\,\,t>0,
\end{equation}
and has a global attractor $\mathcal{A}_\varepsilon\subset X_\varepsilon^\frac{1}{2}$ such that
$$
\sup_{u\in\mathcal{A}_\varepsilon}\|u\|_{L^\infty(\Omega,\R^n)}\leq K,
$$  
for some constant $K$ independent of $\varepsilon$ (see \cite{J.M.Arrieta2000}).  Here $e^{-A_\varepsilon t}$ is the strongly linear semigroup whose infinitesimal generator is $-A_\varepsilon$.

The equation \eqref{limiting_ODE_intro} is well posed problem in $\mathbb{R}^n$ since $F$ is continuous with Lipschitz continuous first derivative. Moreover if we assume that $F$ satisfy the above dissipativeness conditions (ii) then \eqref{limiting_ODE_intro} has solutions defined for all time and  a global attractor $\mathcal{A}_\infty\subset \mathbb{R}^n$.

\section{Asymptotic Behavior}
Once the problem is  well posed in the energy space $X_\varepsilon^\frac{1}{2}$, we prove that in fact the ordinary differential equation \eqref{limiting_ODE_intro}  will describe the asymptotic behaviour of \eqref{reaction_diffusion_equation}. For this we take $\delta>0$ sufficiently small and define the spectral projection $Q_\varepsilon:L^2(\Omega,\mathbb{R}^n)\to L^2(\Omega,\mathbb{R}^n)$,  given by
\begin{equation}\label{projection_reaction_diffusion}
Q_\varepsilon=\frac{1}{2\pi i}\int_{|\xi+1|=\delta}(\xi+A_\varepsilon)^{-1}\,d\xi.
\end{equation} 
Thus the eigenspace $Q_\varepsilon X_\varepsilon^\frac{1}{2}$ is isomorphic to $\R^n$. In fact, the operator $A_\varepsilon$ has compact resolvent and $1\in\sigma(A_\varepsilon)$ is its first eigenvalue, thus $Q_\varepsilon$ is well defined projection with finit rank since $Q_\varepsilon X_\varepsilon^\frac{1}{2}=\tn{span}[\varphi_1^\varepsilon]$, where $\varphi_1^\varepsilon$ is the first eigenfunction of $A_\varepsilon$.

With the aid of the projection $Q_\varepsilon$ we can decompose the phase space $X_\varepsilon^\frac{1}{2}$ in a finite-dimensional subspace and its complement. This decomposition will allow us to decompose the operator $A_\varepsilon$ in order to obtain estimates for the linear semigroup $e^{-A_\varepsilon t}$ restricted to these spaces in the decomposition.

In what follows we denote $L^2=L^2(\Omega,\R^n)$.    

\begin{lemma}\label{Linear_estimate_reaction_diffusion1}
Let $Q_\varepsilon$ be the spectral projection defined in \eqref{projection_reaction_diffusion}. If we denote $Y_\varepsilon=Q_\varepsilon X_\varepsilon^\frac{1}{2}$ and $Z_\varepsilon=(I-Q_\varepsilon)X_\varepsilon^\frac{1}{2}$ and define the projected operators
$$
A_\varepsilon^+=A_\varepsilon|_{Y_\varepsilon}\quad\tn{and}\quad A_\varepsilon^-=A_\varepsilon|_{Z_\varepsilon},
$$
then the following estimates are valid,
\begin{itemize}
\item[(i)] $\|e^{-A^-_\varepsilon  t}z\|_{X_\varepsilon^\frac{1}{2}}\leq Me^{-(d_\varepsilon\lambda_1+1) t}\|z\|_{X_\varepsilon^\frac{1}{2}} ,\quad t> 0,\quad z\in Z_\varepsilon,$
\item[(ii)] $\|e^{-A^-_\varepsilon  t}z\|_{X_\varepsilon^\frac{1}{2}}\leq Me^{-(d_\varepsilon\lambda_1+1) t}t^{-\frac{1}{2}}\|z\|_{L^2} ,\quad t> 0,\quad z\in Z_\varepsilon,$
\end{itemize}
where $-\lambda_1$ is the first nonzero eigenvalue of the Laplacian with homogeneous Neumann boundary conditions on $\Omega$ and $M$ is a constant independent of $d_\varepsilon$. 
\end{lemma}
\begin{proof}
The operator $A_\varepsilon$ is positive and  self adjoint. If we denote its ordered spectrum  $\sigma(A_\varepsilon)=\{1<\lambda_2^\varepsilon<\dots\}$ and $\{\varphi_1^\varepsilon,\varphi_2^\varepsilon,\dots\}$ the associated eigenfunctions, for $z\in Z_\varepsilon$ we have 
$$
e^{-A_\varepsilon^- t}z=e^{-A_\varepsilon t}(I-Q_\varepsilon)z=\sum_{i=2}^\infty e^{-\lambda_i^\varepsilon t}\pin{z,\varphi_i^\varepsilon}_{L^2}\varphi_i^\varepsilon,\quad t>0,
$$
but $\lambda_2^\varepsilon<\lambda_i^\varepsilon$ implies $e^{-\lambda_i^\varepsilon t}<e^{\lambda_2^\varepsilon t}$ for $t>0$. Thus
$$
\|e^{-A^-_\varepsilon  t}z\|_{X_\varepsilon^\frac{1}{2}}\leq \Big(e^{-2\lambda_2^\varepsilon t}\sum_{i=2}^\infty \pin{z,\varphi_i^\varepsilon}^2_{L^2}\lambda_i^\varepsilon\Big)^\frac{1}{2}\leq M e^{-\lambda_2^\varepsilon t}\|z\|_{X_\varepsilon^\frac{1}{2}},\quad t>0.
$$
The function $f(\eta)=e^{-2\eta t}\eta$ attains its maximum at $\eta=1/2t$, $t>0$. Then,
$$
\|e^{-A^-_\varepsilon  t}z\|_{X_\varepsilon^\frac{1}{2}}\leq \begin{cases} e^{-\lambda_2^\varepsilon t}(\lambda_2^\varepsilon)^\frac{1}{2}\|z\|_{L^2}, \,\,\,1/2t<\lambda_2^\varepsilon, \\ e^{-\lambda_2^\varepsilon t}2^{-\frac{1}{2}}t^{-\frac{1}{2}} \|z\|_{L^2},\,\,\,1/2t>\lambda^\varepsilon_2.\end{cases}
$$
The result follows by noticing that $\lambda_2^\varepsilon=d_\varepsilon\lambda_1+1$.
\end{proof}

We can assume, without loss of generality, that $|\Omega|=1$. For the decomposition $X_\varepsilon^\frac{1}{2}=Y_\varepsilon\oplus Z_\varepsilon$ we have $Y_\varepsilon\approx\R^n$ and $Z_\varepsilon=\{\varphi\in X_\varepsilon^\frac{1}{2}:\pin{\psi,\varphi}_{L^2}=0,\,\psi\in Y_\varepsilon\}$, with
$$
\pin{\psi,\varphi}_{L^2}=\int_{\Omega} \varphi(x)\psi(x)\,dx, \quad \varphi\in Y_\varepsilon,\,\,\psi\in Z_\varepsilon.
$$
Since $\psi$ is a constant map, then $\psi\in L^\infty(\Omega,\R^n)$ and thus the above integral is well defined for $\varphi\in X_\varepsilon^\frac{1}{2}$. Hence if $u(t,\cdot)\in X_\varepsilon^\frac{1}{2}$ is a solution of \eqref{reaction_diffusion_equation}, it can be written as $u(t,x)=v(t)+w(t,x)$, where $v\in Y_\varepsilon$ and $w\in Z_\varepsilon$ satisfy    
$$
v(t)=\int_\Omega u(t,x)\,dx\quad\tn{and}\quad \int_\Omega w(t,x)\,dx=0,\,\,\,t>0.
$$
Thus
\begin{align*}
\dot{v}(t)&=\int_\Omega u_t(t,x)\,dx=\int_\Omega E\Delta u(t,x)-u(t,x)\,dx+\int_\Omega F(u(t,x))\,dx\\
&=-v(t)+\int_\Omega F(v(t)+w(t,x))\,dx
\end{align*}
and
\begin{align*}
w_t(t,x)&=u_t(t,x)-\dot{v}(t)\\
&=E\Delta u(t,x)-u(t,x)+F(u(t,x))-\int_\Omega F(v(t)+w(t,x))\,dx+v(t)\\
&=E\Delta w(t,x)-w(t,x)+F(v(t)+w(t,x))-\int_\Omega F(v(t)+w(t,x))\,dx.
\end{align*}
Therefore we can write every solution of \eqref{reaction_diffusion_equation} as a solution of the problem 
\begin{equation}
\begin{cases}\label{couple_ode_limite}
\dot{v}+v=S(v,w),\,\,t>0,\\w_t-E\Delta w+w=Q(v,w),\,\,\,t>0,\,\,x\in\Omega,\\E\frac{\partial w}{\partial\vec{n}}=0,\,\,t>0,\,\,x\in\Gamma,\\w(0)=w_0\in Z_\varepsilon,
\end{cases}
\end{equation}
where
\begin{equation}\label{projection_p_q}
\begin{cases}
S(v,w)=\int_\Omega F(v+w)\,dx,\quad v\in Y_\varepsilon,\,\,w\in Z_\varepsilon,\\Q(v,w)=F(v+w)-\int_\Omega F(v+w)\,dx,\quad v\in Y_\varepsilon,\,\,w\in Z_\varepsilon.
\end{cases}
\end{equation}
It is expected that for $d_\varepsilon$ sufficiently large the part $w(t,x)$ in \eqref{couple_ode_limite} will not play an important role in the asymptotic behavior and, in that case, the limiting equation should  be
\begin{equation}\label{limiting_ODE}
\dot{u}^\infty(t)+u^\infty(t)=F(u^\infty(t)).
\end{equation}   
In fact, the next Theorem inspired by the Theorem 1.1 in \cite{Hale1986} shows that $w(t,x)$ and $g(t,v+w)=F(v+w)-S(v,w)=Q(v,w)$ goes to zero exponentially as $t$ goes to infinity in the energy space $X_\varepsilon^\frac{1}{2}$ when $d_\varepsilon$ is sufficiently large.

\begin{theorem}\label{convergence_exponentially_reaction_diffusion}
Let $S$ and $Q$ as in the definition \eqref{projection_p_q}. Then there is a positive constant $C$ independent of $d_\varepsilon$ such that
$$
\|Q(v(t),w(t))\|_{L^2}\leq Ce^{-(d_\varepsilon\lambda_1+1-\mu)t}\quad\tn{and}\quad \|w(t)\|_{Z_\varepsilon}\leq Ce^{-(d_\varepsilon\lambda_1+1-\mu)t},
$$
where $\mu=(2M\Gamma(\frac{1}{2}))^\frac{1}{2}$ and $M$ is given by the Lemma \ref{Linear_estimate_reaction_diffusion1}.
\end{theorem}
\begin{proof}
Note that $S(v,0)=F(v)$, $Q(v,0)=0$ and $S,Q$ are continuously differentiable with $Q_v(0,0)=0=S_v(0,0)$, thus there is  $\rho>0$ such that for $v_\varepsilon,\tilde{v}_\varepsilon\in Y_\varepsilon$ and $w_\varepsilon,\tilde{w}_\varepsilon\in Z_\varepsilon$,
\begin{itemize}
\item[] $\|Q(v,w)\|_{L^2}\leq \rho,$ 
\item[] $\|Q(v,w)-Q(\tilde{v},\tilde{w})\|_{L^2}\leq \rho(\|z-\tilde{z}\|_{Y_\varepsilon}+\|w-\tilde{w}\|_{Z_\varepsilon})$.
\end{itemize}
Thus
$$
\|Q(v,w)\|_{L^2}\leq \rho \|w\|_{Z_\varepsilon},\quad v\in Y_\varepsilon,\,\,w_\varepsilon\in Z_\varepsilon.
$$ 
Hence we just need to estimate $\|w\|_{Z_\varepsilon}$. 

We use the variation of constants formula to write
$$
w(t)=e^{-A_\varepsilon^- t}w_0+\int_0^t e^{-A_\varepsilon^- (t-s)}Q(v(s),w(s))\,ds. 
$$
Using the estimates from the Lemma \ref{Linear_estimate_reaction_diffusion1}, we have
$$
e^{(d_\varepsilon\lambda_1+1)t}\|w(t)\|_{Z_\varepsilon}\leq M\|w_0\|_{Z_\varepsilon}+M\int_0^t(t-s)^{-\frac{1}{2}}e^{(d_\varepsilon\lambda_1+1)s}\|w(s)\|_{Z_\varepsilon}\,ds,
$$
and by Gronwall's inequality (see \cite{A.N.Carvalho2010} pag 168), we obtain for $\mu=(2M\Gamma(\frac{1}{2}))^\frac{1}{2}$,
$$
\|w(t)\|_{Z_\varepsilon}\leq 2M\|w_0\|_{Z_\varepsilon}e^{-(d_\varepsilon\lambda_1+1-\mu)t}.
$$
\end{proof}

Now we rewrite the ordinary differential equation in \eqref{couple_ode_limite} as $\dot{v}+v=F(v)+[S(v,w)-F(v)]$, it follows from Theorem \ref{convergence_exponentially_reaction_diffusion} that for $d_\varepsilon$ sufficiently large, the asymptotic behavior of \eqref{reaction_diffusion_equation} is determined by the ordinary differential equation \eqref{limiting_ODE}. That is, if $d_\varepsilon\lambda_1>\mu-1$ then the solution $u(t,u_0)$ of the problem \eqref{reaction_diffusion_equation} through $u_0\in X_\varepsilon^\frac{1}{2}$ at $t=0$ satisfies
\begin{equation}\label{exponential_decay_average}
\|u(t,u_0)-v(t)\|_{X_\varepsilon^\frac{1}{2}}\leq K e^{-(d_\varepsilon\lambda_1+1-\mu)t}\overset{t\to\infty}\longrightarrow 0,
\end{equation}  
where $v(t)$ is the average of $u(t,u_0)$ in $\Omega$. Thus, if we assume that the equation \eqref{limiting_ODE} has a global attractor $\mathcal{A}_\infty\subset\R^n$ and understand $\R^n$ as the subspace of constant functions  in $X_\varepsilon^\frac{1}{2}$, we have $A_\infty$ a compact subset in $X_\varepsilon^\frac{1}{2}$ invariant under $T_\varepsilon(\cdot)$ and it follows from \eqref{exponential_decay_average} that $\mathcal{A}_\infty$ attracts under $T_\varepsilon(\cdot)$ bounded set in $X_\varepsilon^\frac{1}{2}$, hence $\mathcal{A}_\infty=\mathcal{A}_\varepsilon$, when $d_\varepsilon\lambda_1>\mu-1$, where $\mu=\sqrt{2M\Gamma(\frac{1}{2})}$. 

\section{Spectral Convergence}
In what follows we prove the convergence of the resolvent operators and we obtain estimates for the spectral projection $Q_\varepsilon$. We establish that the rate for these convergences is $d_\varepsilon^{-\frac{1}{2}}$.  

We saw  that the operators $A_\varepsilon$ and $A_\infty$ work in different spaces. In fact the operator $A_\infty$ is the identity in $\R^n$ that can be understood as the space of constant functions in $X_\varepsilon^\frac{1}{2}$. Thus we need to find a way to compare functions between these spaces. The abstract theory that can be used to compare linear problems in different spaces is developed in \cite{Carvalho2006} and named E-convergence.  In this context we consider the inclusion operator $i:\R^n\to X_\varepsilon^\frac{1}{2}$ and the projection $P:X_\varepsilon^\frac{1}{2}\to \R^n$ given by
$$
Pu=\frac{1}{|\Omega|}\int_\Omega u \,dx,\quad  u\in X_\varepsilon^\frac{1}{2}.
$$
Notice that $P$ can also be considered as an orthogonal projection acting on $L^2$ onto $\R^n$.

The operator $A_\varepsilon$ is  an invertible operator with compact resolvent. The next result shows that the resolvent operator approaches the projection $P$ uniformly in the operator norm. 

\begin{lemma}\label{lemma_rate_reaction_diffusion}
For $g\in L^2(\Omega,\R^n)$ such that $\|g\|_{L^2(\Omega,\R^n)}\leq 1$, let $u^\varepsilon$ be the weak solution of the elliptic problem $A_\varepsilon u^\varepsilon=g$. Then there is a positive constant $C$ independent of $d_\varepsilon$ such that 
\begin{equation}\label{rate_reaction_diffusion}
\|u^\varepsilon-u^\infty\|_{X_\varepsilon^\frac{1}{2}}\leq C d_\varepsilon^{-\frac{1}{2}},
\end{equation}
where $u^\infty=Pg$.
\end{lemma}
\begin{proof}
We denote $u^\varepsilon=(u_1^\varepsilon,...,u_n^\varepsilon)$ and $g=(g_1,...,g_n)$, for $i=1,\dots,n$. Since \eqref{domain_reaction_diffusion_times} holds we can only consider one component $u^\varepsilon_i$. Then
\begin{equation}\label{test_function_reaction_diffusion}
\int_\Omega \varepsilon_i \nabla u^\varepsilon_i\nabla\varphi\,dx + \int_\Omega u^\varepsilon_i\varphi\,dx=\int_\Omega g_i\varphi\,dx,\quad\varphi\in H^1(\Omega);
\end{equation}
$$
\int_\Omega u^\infty_i\psi\,dx=\int_\Omega Pg_i\psi\,dx,\quad \psi\in \R.
$$
Thus
$$
\int_\Omega \varepsilon_i |\nabla u^\varepsilon_i|^2\,dx+\int_\Omega u^\varepsilon_i(u^\varepsilon_i-u^\infty_i)\,dx=\int_\Omega g_i(u^\varepsilon_i-u^\infty_i)\,dx;
$$
$$
\int_\Omega u^\infty_i(Pu^\varepsilon_i-u^\infty_i)\,dx=\int_\Omega Pg_i(Pu^\varepsilon_i-u^\infty_i)\,dx,
$$
which implies
$$
\int_\Omega g_i(u^\varepsilon_i-u^\infty_i)\,dx-\int_\Omega Pg_i(Pu^\varepsilon_i-u^\infty_i)\,dx=\int_\Omega g_i(I-P)u^\varepsilon_i\,dx
$$
and
\begin{align*}
\int_\Omega \varepsilon_i |\nabla u^\varepsilon_i|^2\,dx &+\int_\Omega u^\varepsilon_i(u^\varepsilon_i-u^\infty_i)\,dx-\int_\Omega u^\infty_i(Pu^\varepsilon_i-u^\infty_i)\,dx=\|u^\varepsilon_i-u^\infty_i\|_{X_i^\frac{1}{2}}^2.
\end{align*}
Therefore 
$$
\|u^\varepsilon_i-u^\infty_i\|_{X_i^\frac{1}{2}}^2\leq \int_\Omega |g_i(I-P)u^\varepsilon_i|\,dx.
$$
By Poincaré's inequality for average, we have
$$
\int_\Omega |g_i(I-P)u^\varepsilon_i|\,dx\leq \|g_i\|_{L^2}\Big(\int_\Omega |\nabla u^\varepsilon_i|^2\,dx\Big)^\frac{1}{2},
$$
but
$$
d_\varepsilon\int_\Omega|\nabla u^\varepsilon_i|^2\,dx \leq \|u^\varepsilon_i-u^\infty_i\|_{X_i^\frac{1}{2}}^2.
$$
Put these estimates together we obtain \eqref{rate_reaction_diffusion}. 
\end{proof}

\begin{remark}
When we work with large diffusion the norm in $X_\varepsilon^\frac{1}{2}$ in general is equivalent to the norm of $H^1$ but this equivalence is not uniform, indeed it follows from \eqref{pin_reaction_diffusion} the following inequalities
$$
m_0\|u\|_{H^1}^2\leq \|u\|_{X_\varepsilon^\frac{1}{2}}^2\leq \max_{i=1,...,n}\{\varepsilon_i\} \|u\|_{H^1}^2.
$$
Hence estimates in the Sobolev spaces $H^1$ does not give suitable estimates in the half fractional power space $X_\varepsilon^\frac{1}{2}$, since $d_\varepsilon\leq \max_{i=1,...,n}\{\varepsilon_i\} \to \infty$ as $d_\varepsilon\to \infty$.

Notice that by Poincare's inequality we can obtain a better estimate if we work in $H^1$, that is, $\|u^\varepsilon-u^\infty\|_{H^1}\leq Cd_\varepsilon^{-1}$, for some constant $C$ independent of $d_\varepsilon$.

Hence it is clear that due the non-uniformity in the norms we have some lost when we consider $X_\varepsilon^\frac{1}{2}$-norm than $H^1$-norm. This can be seen in the following example.
\end{remark}

Consider the one-dimensional elliptic problem
$$
\begin{cases}
-\varepsilon u_{xx}=\cos(2\pi x),\,\,x\in(0,1),\\
u_x(0)=0=u_x(1).
\end{cases}
$$
We have $u^\varepsilon(x)=\frac{1}{\varepsilon}\frac{\cos(2\pi x)}{4\pi^2}$ and $u^\infty=0$. Thus
$$
\|u^\varepsilon-u^\infty\|_{X_\varepsilon^\frac{1}{2}}^2=\int_0^1 \varepsilon \Big|\frac{1}{\varepsilon}\frac{\sin(2\pi x)}{4\pi^2}\Big|^2\,dx= C \varepsilon^{-1},
$$
where $C$ is a constant independent of $\varepsilon$.

The Lemma \ref{lemma_rate_reaction_diffusion} determines the natural quantity that will be used to study the  convergence of the dynamic of the problem \eqref{reaction_diffusion_equation} when $\varepsilon$ is approaches $\bar{\mu}=(\mu-1)\lambda_1^{-1}$. The rate of convergence is given by $d_\varepsilon^{-\frac{1}{2}}$ that goes to zero as $d_\varepsilon$ goes to infinity. In fact, if we denote $u^\varepsilon=A_\varepsilon^{-1}g$ then $u^\varepsilon$ is the weak solution of the elliptic problem $A_\varepsilon u^\varepsilon=g$ and since $g$ is an arbitrary map in $L^2$, we obtain
\begin{equation}\label{rate_resolvent}
\|A_\varepsilon^{-1}-P\|_{\LL(L^2,X_\varepsilon^\frac{1}{2})}\leq Cd_\varepsilon^{-\frac{1}{2}}.
\end{equation}
This estimate imply with the compact convergence in \cite{Carvalho2006} and \cite{LPA}, that is the operator $A_\varepsilon^{-1}$ converges compactly to $A_\infty^{-1}P=P$.

Note that, if we take $\varphi=1$ as a test function in \eqref{test_function_reaction_diffusion}, we have $u^\infty=Pu^\varepsilon$, hence \eqref{rate_reaction_diffusion} shows that $u^\varepsilon$ converge for its average in $X_\varepsilon^\frac{1}{2}$ and this rate of convergence is $d_\varepsilon^{-\frac{1}{2}}$. 
 
Now we will see how the converge of the resolvent operators implies the convergence of the eigenvalues and spectral projections defined in \eqref{projection_reaction_diffusion}. We have
\begin{equation}\label{estimate_projection_reaction_diffusion}
\|Q_\varepsilon-P\|_{\LL(L^2,X_\varepsilon^\frac{1}{2})}\leq\frac{1}{2\pi}\int_{|\xi+1|=\delta}\|(\xi+A_\varepsilon)^{-1}-(\xi+I)^{-1}P\|_{\LL(L^2,X_\varepsilon^\frac{1}{2})}\,d\xi\leq C d_\varepsilon^{-\frac{1}{2}}.
\end{equation} 
Since $A_\infty=I$ in $\R^n$ we can denote $Q_\infty=I$, in other words, $Q_\varepsilon$ converges compactly to $Q_\infty^{-1}P=P$. Note that since the operator $A_\varepsilon$ has compact resolvent, the spectral projection $Q_\varepsilon$ is a compact operator. Thus, for $d_\varepsilon$ sufficiently large, the eigenspace $W_\varepsilon=Q_\varepsilon X_\varepsilon^\frac{1}{2}$ has dimension $\tn{dim}(W_\varepsilon)=\tn{dim}(\R^n)=n$. Moreover the eigenvalues $\lambda_i^2$, $i\geq 2$ goes to infinity as $d_\varepsilon$ goes to infinity. The last property was used implicitly in the last section when we guessed the limiting ordinary differential equation. 

\begin{lemma}
Let $A_\varepsilon$ the operator defined in \eqref{operator_reaction_diffusion} and let $\sigma(A_\varepsilon)=\{1<\lambda_2^\varepsilon<\lambda_3^\varepsilon,\dots\}$ its ordered spectrum. Then $\lambda_j^\varepsilon\to \infty$ as $d_\varepsilon\to \infty$ and $j\geq 2$.
\end{lemma}
\begin{proof}
Assume that there is $R>0$ and there are sequences $\varepsilon_k\to \infty$ as $k\to\infty$ and $\{\lambda_j^{\varepsilon_k}\}_k$, $j\geq 2$, such that, $\lambda_j^{\varepsilon_k}\in \sigma(A_{\varepsilon_k})$ and $|\lambda_j^{\varepsilon_k}|\leq R$. We can assume $\lambda_j^{\varepsilon_k}\to \lambda$. Let $u_j^{\varepsilon_k}$ be the corresponding eigenfunction to $\lambda_j^{\varepsilon_k}$ with $\|u_j^{\varepsilon_k}\|_{X_{\varepsilon_k}^\frac{1}{2}}=1$. Then $u_j^{\varepsilon_k}=\lambda_j^{\varepsilon_k}A^{-1}_{\varepsilon_k}u_j^{\varepsilon_k}$. Since $A_{\varepsilon_k}$ converges compactly to $A_\infty^{-1}P$, we can assume  $u_j^{\varepsilon_k}\to u$ as $\varepsilon_k\to \infty$ for some $u\in\R^n$. Thus 
$$
u_j^{\varepsilon_k}=\lambda_j^{\varepsilon_k}A^{-1}_{\varepsilon_k}u_j^{\varepsilon_k}\to \lambda A^{-1}_\infty u,
$$
as $\varepsilon_k\to\infty$. Since $u_j^{\varepsilon_k}\to u$, we get $u=\lambda A^{-1}_\infty u$, which implies $\lambda \in \sigma(A_\infty)$, thus $\lambda=1$ and $\lambda_j^{\varepsilon_k}\to 1$ as $\varepsilon_k\to\infty$, $j\geq 2$, which is an absurd.
\end{proof}

\section{Converge of Attractors}

In what folows we will consider $d_\varepsilon\in[m_0,\bar{\mu}]$, where $\bar{\mu}=(\mu-1)\lambda_1^{-1}$. It is clear that $\mathcal{A}_\varepsilon=\mathcal{A}_\infty$ for $d_\varepsilon\geq \bar{\mu}$, thus we are concerning in what happens when $\varepsilon$ approaches $\bar{\mu}$ to the left. We will see that the family of attractors $\{\mathcal{A}_\varepsilon\}$ with $\varepsilon\in[m_0,\bar{\mu}]$ is continuous as $\varepsilon\to \bar{\mu}$ and this continuity can be estimated by a rate of convergence given by $d_\varepsilon^{-\frac{1}{2}}$ that goes to zero when $d_\varepsilon$ goes to infinity. Since $Y_\varepsilon$ is isomorphic to $\R^n$ and their norms are uniformly equivalent (by \eqref{pin_reaction_diffusion}) we will consider $Y_\varepsilon=\R^n$.

In order to obtain estimate for the convergence of the attractor $\mathcal{A}_\varepsilon$ of the equation \eqref{coupled_system_reaction_diffusion} to the attractor $\mathcal{A}_\infty$ of the \eqref{limiting_ODE} as $d_\varepsilon\to \bar{\mu}$ following the results of the \cite{LPA}, we assume the nonlinear semigroup $T_\infty(\cdot)$ generated by solutions of the \eqref{limiting_ODE} is a Morse-Smale semigroup in $\R^n$. More precisely,  
\begin{equation}\label{vcf2_rd}
T_\infty(t)u^\infty_0=e^{-A_\infty t}u_0^\infty+\int_0^te^{-A_\infty (t-s)}F(T_\infty(s)u_0^\infty)\,ds,\,\,t>0,\,\,\,u^\infty_0\in \R^n,
\end{equation}
where $A_\infty=I$ denote the identity in $\R^n$ and if we denote  $\mathcal{E}_\infty$ the set of its equilibrium points, then it is composed of $p$ hyperbolic points, that is, 
\begin{equation}\label{assumption_unperturbed_1}
\mathcal{E}_\infty=\{\varphi\in \R^n:A_\infty\varphi-F(\varphi)=0\}=\{u_1^{\infty,\ast},\dots,u_p^{\infty,\ast}\},
\end{equation}
where the spectrum set $\sigma(A_\infty-F'(u_i^{\infty,\ast}))\cap \{\varphi\in\R^n:\|\varphi\|_{\R}=1\}=\emptyset$, $i=1,\dots,p$. Moreover, $T_\infty(\cdot)$ is dynamically gradient (see \cite{A.N.Carvalho2010}),
\begin{equation}\label{assumption_unperturbed_2}
\mathcal{A}_\infty=\bigcup_{i=1}^p W^u(u_i^{\infty,\ast}), 
\end{equation}
where $W^u(u_i^{\infty,\ast})$ is the unstable manifold associated to the equilibrium point in $\mathcal{E}_\infty$ and for $i\neq j$ the local unstable manifold $W_{\rm loc}^u(u_i^{\infty,\ast})$ and the stable manifold $W^s(u_j^{\infty,\ast})$ has transversal intersection.  We notice that the Kupka-Smale theorem for ODEs ensures that this situation is generic, in the sense that this must occurs in the most interesting cases.  Thus our assumptions about hyperbolicity and transversality  is not restrictive.

We will study the problem \eqref{reaction_diffusion_equation} as a small perturbation of \eqref{limiting_ODE}
and the continuity  of attractors will be considered, in fact, the assumptions above enable us to obtain the geometric equivalence of phase diagrams when $\varepsilon$ approaches $\bar{\mu}$. This property is known as geometric structural stability and it is the main feature of Morse-Smale problems. In this way we are under the conditions described in \cite{LPA} where results about rate of convergence of attractor for Morse-Smale problems were obtained. More precisely it is valid the following result.

\begin{theorem}
Let $Y_\varepsilon$, $\varepsilon\geq 0$ be a family of separable Hilbert spaces such that $Y_0 \hookrightarrow Y_\varepsilon$ and dim$(Y_0)=n$. Suppose $B_\varepsilon$ is a self adjoint positive and invertible  operator and  consider the following evolution equation
\begin{equation}\label{coupled_system_reaction_diffusion_B}
\begin{cases}
w^\varepsilon_t+B_\varepsilon w^\varepsilon=h(w^\varepsilon),\,\,t>0,\\ w^\varepsilon(0)=w^\varepsilon_0\in Y_\varepsilon^\frac{1}{2},
\end{cases}
\end{equation}
where $Y_\varepsilon^\frac{1}{2}$ is the fractional power space associated with $B_\varepsilon$ $(Y_0^\frac{1}{2}=\mathbb{R}^n)$ and $h$ is a bounded Lipschitz function. Assume that there is a increasing function $\tau(\varepsilon)$ such that $\tau(0)=0$ and 
\begin{equation}\label{rate_resolventB}
\|B_\varepsilon^{-1}-E_\varepsilon B_0^{-1}M_\varepsilon\|_{\LL(Y_\varepsilon,Y_\varepsilon^\frac{1}{2})}\leq C\tau(\varepsilon),
\end{equation}
where $E_\varepsilon:Y_0\to Y_\varepsilon^\frac{1}{2}$ and $M_\varepsilon:Y_\varepsilon \to Y_0$ are bounded linear operators and $C$ is a constant independent of $\varepsilon$. Then there is a invariant manifold for \eqref{coupled_system_reaction_diffusion_B} given by a graph of a Lipschitz function 
$k_\ast^\varepsilon$ such that $\sup_{w^\varepsilon\in Y_0}\|k_\ast^\varepsilon(w^\varepsilon)\|_{X_\varepsilon^\frac{1}{2}}\leq C \tau(\varepsilon)$. Moreover if \eqref{coupled_system_reaction_diffusion_B} with $\varepsilon=0$ generates a Morse-Smale semigroup and if  there is the global attractor $\mathcal{B}_\varepsilon$, for \eqref{coupled_system_reaction_diffusion_B} with $\varepsilon\geq 0$,  then 
$$
\tn{d}_H(\mathcal{B}_\varepsilon,\mathcal{B}_0)\leq C \tau(\varepsilon).
$$
\end{theorem}

Now we can state the main result of this paper.

\begin{theorem}\label{theoLPA}
For $d_\varepsilon \in [m_0,\bar{\mu}]$ there is an invariant manifold $\mathcal{M}_\varepsilon$ for \eqref{coupled_system_reaction_diffusion}, which is given by graph of a certain Lipschitz continuous map $s_\ast^\varepsilon:\R^n\to Z_\varepsilon$ as 
$$
\mathcal{M}_\varepsilon=\{u^\varepsilon\in X_\varepsilon^\frac{1}{2}\,;\, u^\varepsilon = Q_\varepsilon u^\varepsilon+s_{*}^\varepsilon(Q_\varepsilon u^\varepsilon)\}
.$$ 
The map $s_\ast^\varepsilon:\R^n\to Z_\varepsilon$ satisfies the condition
\begin{equation}\label{estimate_invariant_manifold_reaction_diffusion}
|\!|\!|s_\ast^\varepsilon |\!|\!|=\sup_{v^\varepsilon\in \R^n}\|s_\ast^\varepsilon(v^\varepsilon)\|_{X_\varepsilon^\frac{1}{2}}\leq C d_\varepsilon^{-\frac{1}{2}},
\end{equation}
for some positive constant $C$ independent of $d_\varepsilon$.
The invariant manifold $\mathcal{M}_\varepsilon$ is exponentially attracting and the global attractor $\mathcal{A}_\varepsilon$ of the problem \eqref{coupled_system_reaction_diffusion} lying in $\mathcal{M}_\varepsilon$. 
Moreover, the continuity of the attractors can be estimated by 
$$
\tn{d}_H(\mathcal{A}_\varepsilon,\mathcal{A}_\infty)\leq\frac{C}{\sqrt{d_\varepsilon}}.
$$
\end{theorem}
\begin{proof}
If we define $\tau(\varepsilon)=1/ \sqrt{d_\varepsilon}$, then $\tau(\varepsilon)$ is a increasing function such that $$\tau(0)= \lim_{d_\varepsilon\to \infty} 1/ \sqrt{d_\varepsilon} =0.$$
We take $A_0$ as identity in $\mathbb{R}^n$, $E_\varepsilon$ as the inclusion $\mathbb{R}^n\hookrightarrow X^\frac{1}{2}_\varepsilon$ and $M_\varepsilon=P:L^2\to \mathbb{R}^n$ the average in $\Omega$, then by \eqref{rate_resolvent} we have 
$$
\|A_\varepsilon^{-1}-P\|_{\LL(L^2,X_\varepsilon^\frac{1}{2})}\leq C\tau(\varepsilon).
$$
Thus all conditions of the Theorem \eqref{theoLPA}  are satisfied. 
\end{proof}

\bibliographystyle{abbrv}
\bibliography{References}
\end{document}